%% file: kolconj-submitJune29.tex
\documentclass[10pt]{article}
\include{macros}

\usepackage[all]{xy}
\newcommand{\np}[1]{#1}   
\renewcommand{\np}[1]{}

\usepackage[hypertex]{hyperref}

\title{Explicit Heegner Points: Kolyvagin's Conjecture and Non-trivial Elements in the 
Shafarevich-Tate Group} 
\author{Dimitar Jetchev  
\and Kristin Lauter 
\and William Stein 
}

\date{}

\begin{document}
\maketitle

\begin{abstract} 
  Kolyvagin used Heegner points to associate a system of cohomology 
  classes to an elliptic curve over $\Q$ and conjectured that the system contains 
  a non-trivial class. 
  His conjecture has profound implications on the structure of Selmer groups. 
  We provide new computational and theoretical evidence for Kolyvagin's 
  conjecture. More precisely, we explicitly compute Heegner points over ring class 
  fields and use these points to verify the conjecture for specific elliptic curves 
  of rank two. We explain how Kolyvagin's conjecture implies that if the 
  analytic rank of an elliptic curve is at least two then the $\Z_p$-corank of 
  the corresponding 
  Selmer group is at least two as well. We also use explicitly computed Heegner 
  points to produce non-trivial classes in the Shafarevich-Tate group.   
\end{abstract}

\tableofcontents

\section{Introduction}

Let $E_{/F}$ be an elliptic curve over a number field $F$.
The analytic rank $r_\an(E/F)$ of $E$ is 
the order of vanishing  of the $L$-function 
$L(E_{/F}, s)$ at $s = 1$.
The Mordell-Weil rank $r_{\MW}(E/F)$ is the rank of the
Mordell-Weil group $E(F)$.  
The conjecture of Birch and Swinnerton-Dyer is the assertion that $r_\an(E/F) = r_{\MW}(E/F)$. 

Kolyvagin constructed explicit cohomology classes from Heegner points over certain 
abelian extensions of quadratic imaginary fields and used these classes to bound 
the size of the Selmer groups for elliptic curves over $\Q$ of analytic rank at 
most one (see~\cite{kolyvagin:euler_systems},~\cite{kolyvagin:mordellweil} and~\cite{gross:kolyvagin}). His results, together with the Gross-Zagier formula (see~\cite{gross-zagier}), imply the following theorem:  
 
\begin{theorem}[Gross-Zagier, Kolyvagin]
Let $E_{/\Q}$ be an elliptic curve which satisfies $r_{\an}(E/\Q) \leq 1$. Then $r_{\an}(E/\Q) = r_{\MW}(E/\Q)$.
\end{theorem}

Unfortunately, very little is known about the Birch and Swinnerton-Dyer conjecture 
for elliptic curves $E_{/\Q}$ with 
$r_{\an}(E/\Q) \geq 2$. Still, it implies the following conjecture: 

\begin{conjecture}\label{conj:an-MW}
If $r_{\an}(E/\Q) \geq 2$ then $r_{\MW}(E/\Q) \geq 2$. 
\end{conjecture}

As far as we know, nothing has been proved towards the above assertion. 
A weaker conjecture can be formulated in the language of Selmer coranks. 
The Selmer corank $r_p(E/F)$ of $E_{/F}$ is the $\Z_p$-corank of the 
Selmer group $\Sel_{p^\infty}(E/F)$. Using Kummer theory, one shows   
that $r_p(E/\Q) \geq r_{\MW}(E/\Q)$ with an equality occuring if and only 
if the $p$-primary part of the Shafarevich-Tate group $\Sha(E/\Q)$ is 
finite. Thus, one obtains the following weaker conjecture: 

\begin{conjecture}\label{conj:an-corank}
If $r_{\an}(E/\Q) \geq 2$ then $r_{p}(E/\Q) \geq 2$.
\end{conjecture}    

For elliptic curves $E$ of arbitrary analytic rank, Kolyvagin was able 
to explain the exact structure of the Selmer group $\Sel_{p^\infty}(E/\Q)$ 
in terms of 
Heegner points and the associated cohomology classes under a conjecture  
 about 
the non-triviality of these classes (see \cite[Conj.A]{kolyvagin:structure_of_selmer}). Unfortunately, Kolyvagin's conjecture appears to be extremely difficult to prove. 
Until the present paper, 
there has been no example of an elliptic curve over~$\Q$ of rank at least~$2$ for 
which the conjecture has been verified.  

In this paper, we present a complete algorithm to compute  
Kolyvagin's cohomology classes by explicitly computing the corresponding Heegner points 
over ring class fields. We use this algorithm to verify Kolyvagin's conjecture for the 
first time for elliptic curves of analytic rank two. We also explain 
(see Corollary~\ref{cor:large-selmer}) how Kolyvagin's conjecture implies 
Conjecture~\ref{conj:an-corank}. In addition, we use methods of Christophe Cornut 
(see~\cite{cornut:compte-rendus}) to provide theoretical evidence for Kolyvagin's conjecture. Finally, as a separate application of the explicit computation of Heegner
points, we construct nontrivial cohomology classes in the
Shafarevich-Tate group $\Sha(E/K)$ of elliptic curves $E$ over certain quadratic 
imaginary fields. 

The paper is organized as follows. Section~\ref{sec:heegpts}
introduces Heegner points over ring class fields and Kolyvagin
cohomology classes. We explain the method of computation and
illustrate them with two examples. In Section~\ref{sec:kolconj} we
state Kolyvagin's conjecture, discuss Kolyvagin's work on Selmer
groups and establish Conjecture~\ref{conj:an-corank} as a corollary. Moreover, 
we present a proof of the theoretical evidence following closely 
Cornut's arguments.
Section~\ref{subsec:389A1} contains the essential examples for which
we manage to explicitely verify the conjecture. Finally, in
Section~\ref{sec:expl-sha} we apply our computational techniques to
produce explicit non-trivial elements in the Shafarevich-Tate 
groups for specific elliptic curves.
        
\vspace{2ex}
{\bf Acknowledgment.} We are indebted to Christophe Cornut for sharing his thoughts 
on the Kolyvagin's conjecture. We would like to thank Jan Nekov\'{a}\v{r} for his 
helpful discussions on height bounds, Stephen D. Miller and Mark Watkins for 
pointing out analytic methods for estimating special values of derivatives of 
automorphic $L$-functions, Henri Darmon and David Jao for useful conversations on 
Heegner points computations, and Ken Ribet for reading the preliminary draft and for 
useful conversations.

\section{Heegner points over ring class fields}\label{sec:heegpts}
We discuss Heegner points over ring class fields in Section~\ref{subsec:heegdefs} and describe a method for computing them in Section~\ref{subsec:heegptscomp}. Height 
estimates for these points are given in the appendix. We illustrate the method with some examples in Section~\ref{subsec:ptsexamples}. The standard references are \cite{gross:kolyvagin}, \cite{kolyvagin:euler_systems} and~\cite{mccallum:kolyvagin}.

\subsection{Heegner points over ring class fields}\label{subsec:heegdefs}
Let $E$ be an elliptic curve over $\Q$ of conductor $N$ and let $K =
\Q(\sqrt{-D})$ for some fundamental discriminant $-D < 0$, $D \neq 3, 4$, such that 
all prime factors of $N$ are split in $K$. We refer to such a discriminant as a  
\emph{Heegner discriminant} for $E/\Q$. Let $\cO_K$ be the ring of integers of $K$.  
It follows that $N\cO_K = \cN\bar{\cN}$ for an ideal $\cN$ of $\cO_K$ with 
$\cO_K / \cN \simeq \Z / N\Z$. 

By the modularity theorem (see~\cite{breuil-conrad-diamond-taylor}), there exists 
a modular parameterization $\varphi : X_0(N) \ra E$. 
Let $\cN^{-1}$ be the
fractional ideal of $\cO_K$ for which $\cN \cN^{-1} = \cO_K$.
We view $\cO_K$ and $\cN$ as
$\Z$-lattices of rank~$2$ in $\C$ and observe that $\C / \cO_K \ra \C /
\cN^{-1}$ is a cyclic isogeny of degree $N$ between the elliptic
curves $\C / \cO_K$ and $\C / \cN^{-1}$.  This
isogeny corresponds to a complex point $x_1 \in X_0(N)(\C)$. According to
the theory of complex multiplication~\cite[Ch.II]{silverman:aec2}, the
point $x_1$ is defined over the Hilbert class field $H_K$ of $K$.

More generally, for an integer $c$, 
let $\cO_c = \Z + c\cO_K$ be the order of conductor $c$ in
$\cO_K$ and let $\cN_c = \cN \cap \cO_c$, which is an invertible ideal of $\cO_c$. 
Then $\cO_c / \cN_c \simeq \Z / N\Z$ and the map $\C / \cO_c \ra \C / \cN_c^{-1}$ 
is a cyclic isogeny of degree $N$. Thus, it defines a point $x_c \in X_0(N)(\C)$.
By the theory of complex multiplication, this point is defined
over the ring class field $K[c]$ of conductor $c$ over $K$ (that is, the 
unique abelian extension of $K$ corresponding to the norm subgroup 
$\widehat{\cO_c}^\times K^\times \subset \widehat{K}^\times$; e.g., if $c = 1$ then 
$K[1] = H_K$). 
 
We use the parameterization $\varphi : X_0(N) \ra E$ to obtain points
$$
  y_c = \varphi (x_c) \in E(K[c]).
$$ 
Let $y_K = \Tr_{H_K / K}(y_1)$. We refer to $y_K$ as the
{\em Heegner point} for the discriminant~$D$, even though it is only
well defined up to sign and torsion (if $\cN'$ is another ideal with 
$\cO / \cN' \simeq \Z / N\Z$ then the new Heegner point differs from $y_K$ by at most 
a sign change and a rational torsion point). 

\subsection{Explicit computation of the points $y_c$}\label{subsec:heegptscomp}
Significant work has been done on explicit calculations of Heegner points on 
elliptic curves (see~~\cite{cohen:nt},~\cite{delauney:phd},
~\cite{elkies:heegner},~\cite{watkins:heegner}). Yet, all of these compute only 
the points $y_1$ and $y_K$. In~\cite{ejl:cassels} explicit computations of the 
points $y_c$ were considered in several examples and some difficulties were outlined.  

To compute the point $y_c = [\C / \cO_c \ra \C / \cN_c^{-1}] \in E(K[c])$ we let $f \in S_2(\Gamma_0(N))$ be the newform corresponding to the elliptic curve $E$ and $\Lambda$ be 
the complex lattice (defined up to homothety), such that $E \isom \C / \Lambda$. Let 
$\mathfrak h^\times = \mathfrak h \cup \mathbb{P}^1(\Q) \cup \{i \infty\}$, where $\mathfrak h = \{z \in \C : \Im(z) > 0\}$, equipped with the action 
of $\Gamma_0(N)$ by linear fractional transformations. The modular parametrization $\vphi : X_0(N) \ra E$ is then given by the function 
$\varphi : \h^\times \ra \C / \Lambda$ 
\begin{equation}\label{eqn:modparam}
\varphi(\tau) = \int_{\tau}^{i \infty} f(z) dz 
      = \sum_{n \geq 1} \frac{a_n}{n}e^{2\pi i n \tau},    
\end{equation}
where $\ds f = \sum_{n = 1}^\infty a_n q^n$ is the Fourier expansion of the modular 
form $f$. 

We first compute ideal class representatives $\mathfrak a_1, \mathfrak a_2, \dots, \mathfrak a_{h_c}$ for the Picard group 
$\Pic(\cO_c) \isom \Gal(K[c]/K)$, where $h_c = \# \Pic(\cO_c)$. Let $\sigma_i \in \Gal(K[c]/K)$ be the image of the the ideal class of $\mathfrak a_i$ under the Artin map. Thus, we can use the ideal $\mathfrak a_i$ to compute a complex number $\tau_i \in \h$ representing the CM point $\sigma_i(x_c)$ for each $i = 1, \dots, h_c$ (since 
$X_0(N) = \Gamma_0(N) \backslash \h^\times$). Explicitly, the Galois conjugates of 
$x_c$ are 
$$
\sigma_i(x_c) = [\C / \mathfrak a_i^{-1} \ra \C / \mathfrak a_i^{-1} \cN_c^{-1}],\ \forall i = 1, \dots, h_c.  
$$
 Next, we can use (\ref{eqn:modparam}) to approximate $\vphi(\sigma_i(x_c))$ as an element 
of $\C / \Lambda$ by truncating the infinite series. Finally, the image of $\vphi(\tau_i) + \Lambda$ under the Weierstrass $\wp$-function 
gives us an approximation of the $x$-coordinate of the point $y_c$ on the Weierstrass model of the elliptic curve $E$. On the other hand, 
this coordinate is $K[c]$-rational. Thus, if we compute the map (\ref{eqn:modparam}) with 
sufficiently many terms and up to high enough floating point accuracy, we must 
be able to recognize the correct $x$-coordinate of $y_c$ on the Weierstrass model as an element of 
$K[c]$. 

To implement the last step, we use the upper bound established  
on the logarithmic height 
of the Heegner point $y_c$ (given in the appendix). The bound on the logarithmic height 
comes from a bound on the canonical height combined with bounds on the height difference 
(see the appendix for complete details). 
Once we have a height bound, we estimate the floating point accuracy required for the computation. Finally, we estimate the number of terms 
of (\ref{eqn:modparam}) necessary to compute the point $y_c$ up to the corresponding 
accuracy (see~\cite[p.591]{cohen:nt}).   

\begin{remark}
In practice, there are two ways to implement the above algorithm. The first approach is to compute an approximation $x_i$ of the $x$-coordinates of 
$y_c^{\sigma_i}$ for every $i = 1, \dots, c$ and form the polynomial $F(z) = \prod_{i = 1}^{h_c} (z - x_i)$. 
The coefficients of this polynomial are very close to the rational coefficients of the minimal polynomial of the actual $x$-coordinate of $y_c$.  
Thus, one can try to recognize the coefficients of $F(z)$ by using the continued fractions method. The second approach is to search for the $\tau_i$ with 
the largest imaginary part (which will make the convergence of the corresponding series (\ref{eqn:modparam}) defining the modular parametrization fast) and then 
try to search for an algebraic dependence of degree $[K[c] : K]$ using standard algorithms implemented in PARI/GP. Indeed, computing a conjugate with a smaller imaginary part 
might be significantly harder since the infinite series in (\ref{eqn:modparam}) will 
converge slower and one will need more terms to compute the image up to the 
required accuracy.     
\end{remark}   

\begin{remark}
We did {\em not} actually implement an algorithm for computing bounds
on heights of Heegner points as described in the appendix of this
paper.  Thus the computations below are not provably correct, though
we did many
consistency checks, and our computational observations
are almost certainly correct.  The primary goal of the examples
and practical implementation of our algorithm is to provide
tools and data for improving our theoretical understanding of
Kolyvagin's conjecture, and not making the computations below
provably correct does not detract from either of these goals.
\end{remark}
  
\subsection{Examples}\label{subsec:ptsexamples}
We compute the Heegner points $y_c$ for specific elliptic curves and choices of quadratic imaginary fields. 

\vspace{0.1in} 
 
\noindent {\bf \rm{\bf 53A1}:} Let ${E}_{/\Q}$ be the elliptic curve with label \textrm{\bf 53A1} in Cremona's database. Explicitly, $E$ is the curve
{\small   
$y^2 + xy + y = x^3 - x^2.$
}
\noindent Let $D = 43$ and $c = 5$. The conductor of $E$ is $53$ which is split 
in $K = \Q(\sqrt{D})$, so $D$ is a Heegner discriminant for $E$. The modular form associated 
to $E$ is   
{\small 
$
f_{E}(q) = q - q^2  - 3q^3 - q^4 + 3q^6 - 4q^7 + 3q^8 + 6q^9 + \cdots . 
$
}
\noindent One applies the methods from Section~\ref{subsec:heegptscomp} to compute the minimal polynomial of the $x$-coordinate 
of $y_5$ for the above model 
{\small   
$$
F(x) = x^6 - 12x^5 + 1980x^4 - 5855x^3 + 6930x^2 - 3852x + 864.
$$
}
\noindent Since $F(x)$ is an irreducible polynomial over $K$, it generates the ring class field $K[5]/K$, i.e., 
$K[5] = K[\alpha] \isom K[x] / \langle F(x) \rangle$, where $\alpha$ is one of the roots. To find the $y$-coordinate of $y_5$ we substitute $\alpha$ into the equation of 
$E$ and factor the resulting quadratic polynomial over $K[5]$ 
to obtain that the point $y_5$ is equal to   
{\small
$$
\left( \alpha , -4/315\alpha^5 + 43/315 \alpha^4 - 7897/315\alpha^3 + 2167/35 \alpha^2 
- 372/7\alpha + 544/35 \right ) \in E(K[5]).   
$$ 
}

\vspace{0.1in}

\noindent {\bf \rm{\bf 389A1:}}
The elliptic curve with label \textrm{\bf 389A1} is
{\small 
$
y^2 + y = x^3 + x^2 - 2x
$
}
and the associated modular form   
{\small 
$
f_{E}(q) = q - 2q^2 -2q^3 +2q^4 -3q^5 + 4q^6 -5q^7 +q^9 + 6q^{10} + \cdots. 
$
}
Let $D = 7$ (which is a Heegner discriminant for $E$) and $c = 5$. As above, we 
compute the minimal polynomial of the $x$-coordinate of $y_{5}$
{\small 
$$
F(x) = x^6 + \frac{10}{7}x^5 - \frac{867}{49}x^4 - \frac{76}{245}x^3 + \frac{3148}{35}x^2 - 
\frac{25944}{245}x + \frac{48771}{1225}. 
$$
}
If $\alpha$ is a root of $F(x)$ then $y_5 = (\alpha,\beta)$ where 
{\small
\begin{eqnarray*}
\beta &=& \frac{280}{7761}\sqrt{-7}\alpha^5 + \frac{1030}{7761}\sqrt{-7}\alpha^4 - 
\frac{12305}{36218}\sqrt{-7}\alpha^3 - \frac{10099}{15522}\sqrt{-7}\alpha^2 \\ 
&+& \frac{70565}{54327}\sqrt{-7}\alpha + \frac{- 18109 -33814\sqrt{-7} }{36218}. 
\end{eqnarray*} 
}

\vspace{0.1in}

\noindent {\bf \rm{\bf 709A1:}} The curve \textrm{\bf 709A1} with equation  
{\small 
$
y^2 + y = x^3 - x^2 - 2x
$
}
has associated modular form  
{\small 
$
f_{E}(q) =  q - 2q^2 - q^3 + 2q^4 - 3q^5 + 2q^6 - 4q^7 - 2q^9 + \cdots. 
$
}
Let $D = 7$ (a Heegner discriminant for $E$) and $c = 5$. The minimal polynomial 
of the $x$-coordinate of $y_5$ is  
{\small 
$
F(x) = \frac{1}{5^2\cdot 7^2 \cdot 19^2} \left ( 442225x^6 - 161350x^5 - 2082625x^4 - 387380x^3 + 2627410x^2 + 18136030x + 339921 \right ),  
$
}
and if $\alpha$ is a root of $x$ then $y_5 = (\alpha, \beta)$
{\small 
\begin{eqnarray*}
\beta &=& \frac{341145}{62822}\sqrt{-7}\alpha^5 - \frac{138045}{31411}\sqrt{-7}\alpha^4 - \frac{31161685}{1319262}\sqrt{-7}\alpha^3 + 
\frac{7109897}{1319262}\sqrt{-7}\alpha^2 + \\ 
&+& \frac{39756589}{1319262}\sqrt{-7}\alpha + \frac{- 219877+ 4423733\sqrt{-7}}{439754}. 
\end{eqnarray*}
}
\vspace{0.1in}

\noindent {\bf \rm{\bf 718B1:}} The curve \textrm{\bf 718B1} has equation 
{\small 
$
\ y^2 + xy + y = x^3 - 5x
$
}
with associated modular form 
{\small 
$
f_E(q) = q - q^2 - 2q^3 + q^4 - 3q^5 + 2q^6 - 5q^7 - q^8 + q^9 + 3q^{10} + \dots. 
$
}
Again, for $D = 7$ and $c = 5$ we find    
{\small 
$
F(x) = \frac{1}{3^4\cdot 5^2}\left (2025x^6 + 12400x^5 + 32200x^4 + 78960x^3 + 289120x^2 + 622560x + 472896 \right )  
$
}
and $y_5 = (\alpha, \beta)$ with 
{\small 
\begin{eqnarray*}
\beta &=& \frac{16335}{12271}\sqrt{-7}\alpha^5 + \frac{206525}{36813}\sqrt{-7}\alpha^4 + \frac{54995}{5259}\sqrt{-7}\alpha^3 + 
\frac{390532}{12271}\sqrt{-7}\alpha^2 + \\
&+& \frac{- 36813 + 9538687\sqrt{-7}}{73626}\alpha + \frac{- 12271 + 4018835\sqrt{-7}}{24542}. 
\end{eqnarray*}
}

\section{Kolyvagin's conjecture: consequences and evidence}\label{sec:kolconj} 
We briefly recall Kolyvagin's construction of the cohomology classes in Section~\ref{subsec:kolclasses} and state 
Kolyvagin's conjecture in Section~\ref{subsec:kolconj}. Section~\ref{subsec:selmer} is devoted to the proof of the promised 
consequence regarding the $\Z_p$-corank of the Selmer group of an elliptic curve with 
large analytic rank. In Section~\ref{subsec:cornut} we provide 
Cornut's arguments for the theoretical evidence for Kolyvagin's conjecture and finally, in Section~\ref{subsec:389A1} we verify Kolyvagin's conjecture 
for particular elliptic curves. Throughout the entire section we assume that 
$E_{/\Q}$ is an elliptic curve of conductor $N$, $D$ is a Heegner discriminant for $E$ and $p \nmid ND$ is a prime such that the mod $p$ Galois representation 
$\rhobar_{E, p} : \Gal(\Qbar / \Q) \ra \Aut(E[p])$ is surjective. 

\subsection{Preliminaries}
Most of this section follows the exposition in~\cite{gross:kolyvagin},~\cite{mccallum:kolyvagin} and~\cite{kolyvagin:structureofsha}.

\vspace{0.1in}

\noindent \textit{1. Kolyvagin primes.} We refer to a prime number $\ell$ as a \emph{Kolyvagin prime} if $\ell$ is inert in
$K$ and $p$ divides both $a_\ell$ and $\ell+1$). For a Kolyvagin prime $\ell$ let
$$
  M(\ell) = \ord_p(\gcd(a_\ell, \ell+1)).
$$ 
We denote by $\Lambda^r$ the set of all square-free products of 
exactly $r$ Kolyvagin primes and let 
$\ds \Lambda = \bigcup_r \Lambda^r$. 
For any $c \in \Lambda$, let 
$\ds M(c) = \min_{\ell \mid c} M(\ell)$. Finally, let 
$$
\Lambda^r_m = \{c \in \Lambda^r : M(c) \geq m\}
$$ 
and let $\ds \Lambda_m = \bigcup_{r} \Lambda_m^r$. 

\vspace{0.1in}

\noindent \textit{2. Kolyvagin derivative operators.} 
Let $\cG_c= \Gal(K[c]/K)$ and $G_c = \Gal(K[c]/K[1])$. 
For each $\ell \in \Lambda^1$, the group $G_\ell$ is cyclic of order 
$\ell + 1$. Indeed, 
$$
G_\ell \simeq (\cO_K / \ell \cO_K)^\times / (\Z / \ell \Z)^\times \simeq \F_\lambda^\times / \F_\ell^\times. 
$$
Moreover, $\ds G_c \isom \prod_{\ell \mid c} G_\ell$ (since 
$\Gal(K[c]/K[c/\ell])\isom G_\ell$). 
Next, fix a generator $\sigma_\ell$ of $G_\ell$ for each $\ell \in \Lambda^1$. Define 
$D_\ell = \sum_{i = 1}^\ell i \sigma_\ell^i \in \Z[G_\ell]$ and let 
$$
D_c = \prod_{\ell \mid c} D_\ell \in \Z[G_c]. 
$$
Note that $(\sigma_\ell - 1) D_\ell = 1 + \ell - \Tr_{K[\ell]/K[1]}$. 

We refer to $D_c$ as the \emph{Kolyvagin derivative operators}. Finally, let $S$ be a set of coset representatives 
for the subgroup $G_c \subseteq \cG_c$. Define 
$$
P_c = \sum_{s \in S} s D_c y_c \in E(K[c]).  
$$ 
The points $P_c$ are derived from the points $y_c$, so we will refer to them as \emph{derived Heegner points}. 

\vspace{0.1in}

\noindent \textit{3. The function $m : \Lambda \ra \Z$ and the sequence $\ds \{m_r\}_{r \geq 0}$.}
For any $c \in \Lambda$ let $m'(c)$ be the largest positive integer such that
$P_c \in p^{m'(c)}E(K[c])$ (if $P_c$ is torsion then $m'(c) = \infty$). 
Define a function $m : \Lambda \ra \Z$ by 
$$
m(c) = \left \{ 
\begin{array}{ll} 
m'(c) & \textrm{ if }m'(c) \leq M(c), \\
\infty & \textrm{ otherwise.}
\end{array} \right . 
$$ 
Finally, let $\ds m_r = \min_{c \in \Lambda^r} m(c)$. 

\begin{proposition}\label{prop:inequal}
The sequence $\{m_r\}_{r \geq 0}$ is non-increasing, i.e., 
$m_r \geq m_{r+1}$.  
\end{proposition}

\begin{proof}
This is proved in~\cite[Thm.C]{kolyvagin:structureofsha}.
\end{proof}

\subsection{Kolyvagin cohomology classes}\label{subsec:kolclasses}
Kolyvagin uses the points $P_c$ to construct classes $\kappa_{c, m} \in \H^1(K, E[p^m])$ for any $c \in \Lambda_m$. For the details of the construction, 
we refer to~\cite[pp.241-242]{gross:kolyvagin}) and~\cite[\S 4]{mccallum:kolyvagin}. 
The class $\kappa_{c, m}$ is explicit, in the sense that it is represented by the 1-cocycle
\begin{eqnarray}\label{eqn:cocycle}
\sigma \mapsto \sigma \left (\frac{P_c}{p^m} \right ) - \frac{P_c}{p^m} - \frac{(\sigma - 1)P_c}{p^m}, 
\end{eqnarray}
where $\ds \frac{(\sigma - 1)P_c}{p^m}$ is the unique $p^m$-division point of $(\sigma-1)P_c$ in $E(K[c])$ (see~\cite[Lem.~4.1]{mccallum:kolyvagin}). 
The class $\kappa_{c, m}$ is non-trivial if and only if $P_c \notin p^m E(K[c])$ 
(which is equivalent to $m > m(c)$). 

Finally, let $-\eps$ be the sign of the functional equation corresponding to $E$. 
For each $c \in \Lambda_m$, let 
$\eps(c) = \eps \cdot (-1)^{f_c}$ where $f_c = \# \{\ell : \ell \mid c\}$ 
(e.g., $f_1 = 0$). It follows from~\cite[Prop.5.4(ii)]{gross:kolyvagin} that 
$\kappa_{c, m}$ lies in the $\eps(c)$-eigenspace for the action of complex conjugation on $\H^1(K, E[p^m])$.

\subsection{Statement of the conjecture}\label{subsec:kolconj}
We are interested in $\ds m_\infty = \min_{c \in \Lambda} m(c) = \lim_{r \ra \infty} m_r$. In the case when 
the Heegner point $P_1 = y_K$ has infinite order in $E(K)$, the Gross-Zagier formula 
(see~\cite{gross-zagier}) implies that $E(K)$ has rank~$1$, i.e., $m_0 < \infty$ as it is $\ord_p([E(K) : \Z y_K])$. 
In that case, $m_\infty < \infty$, so the system of cohomology classes 
$$
  T = \{\kappa_{c, m} : m \leq M(c)\}
$$ 
is nonzero.
A much more interesting and subtle is the case of 
an elliptic curves $E$ over $K$ of rank at least $2$. 
Kolyvagin conjectured (see~\cite[Conj.C]{kolyvagin:structure_of_selmer}) that in 
all cases $T$ is non-trivial.  

\begin{conjecture}[Kolyvagin's conjecture]\label{conj:kolconj}
We have $m_\infty < \infty$, i.e., $T$ is non-trivial.
\end{conjecture}

\begin{remark}
Kolyvagin's conjecture is obvious in the case of elliptic curves of analytic rank 
one over $K$ since $m_0 < \infty$ (which follows from Gross-Zagier's formula). Still, 
it turns out that the $p$-part of the Birch and Swinnerton-Dyer conjectural formula 
is equivalent to $\ds m_\infty = \ord_p \left (\prod_{q \mid N} c_q \right )$, where 
$c_q$ is the Tamagawa number of $E_{/\Q}$ at $q$. See~\cite{jetchev:tamagawa} for some 
new results related to this question which imply (in many cases) the exact upper bounds 
on the $p$-primary part of the Shafarevich-Tate group as predicted by the BSD formula. 
\end{remark}

\subsection{A consequence on the structure of Selmer groups}\label{subsec:selmer}

\begin{theorem}[Kolyvagin]\label{thm:kolselstruct}
Assume Conjecture~\ref{conj:kolconj} and let $f$ be the smallest nonnegative 
integer for which $m_f < \infty$. Then 
$$
\Sel_{p^\infty}(E/K)^{\eps(-1)^{f+1}} \isom (\Q_p / \Z_p)^{f+1} 
\oplus \textrm{(a finite group)}
$$
and 
$$
\Sel_{p^\infty}(E/K)^{\eps(-1)^f} \isom (\Q_p/\Z_p)^r \oplus 
   \textrm{(a finite group)}  
$$ 
where $r \leq f$ and $f - r$ is even.
\end{theorem}

The above structure theorem of Kolyvagin has the following consequence which strongly supports Conjecture~\ref{conj:an-corank}. 

\begin{corollary}\label{cor:large-selmer}
Assume Conjecture~\ref{conj:kolconj}. Then 
\noindent (i) If $r_{\an}(E/\Q)$ is even and nonzero then  
$$
r_{p}(E/\Q) \geq 2.
$$  

\noindent (ii) If $r_\an(E/\Q)$ is odd and strictly larger than $1$ then 
$$
r_p(E/\Q) \geq 3.
$$ 
\end{corollary}

\begin{proof}
(i) By using~\cite{bump-friedberg-hoffstein} or~\cite{murty-murty} one can choose a quadratic imaginary field $K = \Q(\sqrt{-D})$, such 
that the derivative $L'(E^D_{/\Q}, s)$ of the $L$-function $L(E^D_{/\Q}, s)$ of the twist $E^D$ of $E$ by the 
quadratic character associated to $K$ does not vanish at $s = 1$. This means (by Gross-Zagier's 
formula~\cite{gross-zagier}) that the basic Heegner point $y_K$ has infinite order and thus, by Kolyvagin's work, 
the Selmer group $\Sel_{p^\infty}(E^D/\Q)$ has corank one, i.e., $r_p^{-}(E/K) = 1$. We want to show that 
$r_p(E/K) \geq 3$, i.e., $r_p^{+}(E/K)  = r_p(E/\Q)\geq 2$. Assume the contrary, i.e. $r_p^{+}(E/K) \leq 1$. 
Then, according to Theorem~\ref{thm:kolselstruct}, $r = 0$. Since $f$ has the same parity as $r$, we conclude 
that $f = 0$ as well, i.e., the Heegner point $y_K$ has infinite order in 
$E(K)$ and hence (by the Gross-Zagier formula) the $L$-function vanishes 
to order 1 which is a 
contradiction, since by hypothesis $r_{\an}(E/\Q) > 0$.
Therefore $r_p(E/\Q) = r_p^{+}(E/K) \geq 2$. \\

\noindent (ii) It follows from the work of Waldspurger (see also~\cite[pp.543-44]{bump-friedberg-hoffstein}) that 
one can choose a quadratic imaginary field $K = \Q(\sqrt{-D})$, such that the $L$-function of the twist $E^D$ 
satisfies $L(E^D, 1) \ne 0$. This means that $r_p(E^D/\Q) = 0$, i.e., $r_p^+(E/K) = 0$. Thus, by Theorem~\ref{thm:kolselstruct} we obtain 
$r = 0$ and $f$ is even ($r$ and $f$ are as in Theorem~\ref{thm:kolselstruct}). If $f > 0$ we are done because in that case $r_p(E/K) \geq 3$. 
If $f = 0$, we use the same argument as in (i) to arrive at a contradiction. Therefore, $$r_p(E/\Q) = r_p^{+}(E/K) \geq 3.$$
\end{proof}

\subsection{Cornut's theoretical evidence for Kolyvagin's conjecture}\label{subsec:cornut}
The following evidence for Conjecture~\ref{conj:kolconj} was proven by Christophe Cornut. 

\begin{proposition}\label{prop:evidence}
For all but finitely many $c \in \Lambda$ there exists a choice 
$R$ of liftings for the elements of 
$\Gal(K[1]/K)$ into $\Gal(K^{\ab}/K)$, such that if $P_c = D_0 D_c y_c$ is the Heegner point defined in 
terms of this choice of liftings (i.e, if $\ds D_0 = \sum_{\sigma \in R} \sigma$), then 
$P_c$ is non-torsion.
\end{proposition}

\begin{remark}
For a nontorsion point $P_c$, let $e_c$ denotes the minimal exponent $e$, such that 
$P_c \notin p^{e_c}E(K[c])$. 
Proposition~\ref{prop:evidence} gives  very  little evidence towards the Kolyvagin conjecture. 
The reason is that even if one gets non-torsion points $P_c$, it might 
still happen that for each such $c$ we have $e_c > M(c)$ in which case all classes  
$\kappa_{c, m}$ with $m \leq M(c)$ will be trivial.  
\end{remark}

Let $\ds K[\infty] = \bigcup_{c \in \Lambda} K[c]$.   
\begin{lemma}\label{lem:fintors}
The group $E(K[\infty])_{\tors}$ is finite. 
\end{lemma}

\begin{proof}
Let $q$ be any prime which is a prime of good reduction for $E$, which is inert in $K$ and which is different from the primes in 
$\Lambda^1$. Let $\mathfrak q$ be the unique prime of $K$ over $q$. It follows from class field theory that the prime $\mathfrak q$ splits 
completely in $K[\infty]$ since it splits in each of the finite extensions $K[c]$. Thus, the completion of $K[\infty]$ at any prime which 
lies over $\ell$ is isomorphic to $K_{\mathfrak q}$ and therefore, $E(K[\infty])_{\tors} \hra E(K_\lambda)_{\tors}$. The last group is finite 
since it is isomorphic to an extension of $\Z_q^2$ by a finite group (see~\cite[Lem.I.3.3]{milne:duality} or~\cite[p.168-169]{tate:p-div}). 
Therefore, $E(K[\infty]_{\tors})$ is finite. 

\comment{
Let $\ell \notin \Lambda^1$ be an arbitrary prime which is inert in $K$ (take any inert prime $\ell$, such that 
$p \nmid \ell+1$) and let $\lambda \mid \ell$ be the unique prime of $K$ above $\ell$. It follows from class field theory that the prime  $\ell$ splits completely in $K[\infty]$ since it splits completely in each of the 
finite extensions $K[n]$.\edit{I don't understand this at all.  Maybe this would
be right if $\ell$ were in $\Lambda^1$...}
Thus, the completion of $K[\infty]$ at any prime which lies over $\ell$ is 
isomorphic to 
$K_\lambda$ and therefore, 
$E(K[\infty])_{\tors} \hra E(K_\lambda)_{\tors}$ which is finite.  
\edit{We should give a reference that $E(K_\lambda)_{\tors}$ is
finite, e.g., using formal groups.}
}
\end{proof}

Let $|E(K[\infty])_{\tors}| = M < \infty$ and let $\ds d(c) = \prod_{\ell \mid c} (\ell+1)$ for any $c \in \Lambda$. 
Let $m_E$ be the modular degree of $E$, i.e., the degree 
of an optimal modular parametrization $\pi : X_0(N) \ra E$. 

\begin{lemma}\label{lem:D0}
Suppose that $c \in \Lambda$ satisfies $d(c) > m_E M$. There exists a lifting $R$ of $\Gal(K[1] / K)$ in 
$\Gal(K[c]/K)$, such that $D_0 y_c \notin E(K[c])_{\tors}$, where $\ds D_0 = \sum_{\sigma \in R} \sigma$. 
\end{lemma}

\begin{proof}
The $\Gal(K[c]/K[1])$-orbit of the point $x_c \in X_0(N)(K[c])$ consists of $d(c)$ distinct points, so there are 
at least $d(c) / m_E$ elements in the orbit $\Gal(K[c]/K[1])y_c$. Choose a set of representatives $R$ 
of $\Gal(K[c]/K) / \Gal(K[c]/K[1])$ which contains the identity element 
$1 \in \Gal(K[c]/K)$. For $\tau \in \Gal(K[c]/K[1])$
define 
$$
R_\tau = (R -\{\sigma_0\}) \cup \{\tau\}.  
$$  
Let $\ds S = \sum_{\sigma \in R} \sigma y_c$ and $\ds S_\tau = \sum_{\sigma \in R_\tau} \sigma y_c$. Then 
$$
S_\tau - S = \sigma y_c - y_c,   
$$
which takes at least $d(c) / m_E > M$ distinct values. Therefore, there exists an 
automorphism 
$\tau \in \Gal(K[c]/K[1])$, for which 
$S_\tau \notin E(K[c])_{\tors}$, which  proves the lemma.   
\end{proof}

\begin{proof}[Proof of Proposition~\ref{prop:evidence}]
Suppose that $c \in \Lambda$ satisfies the statement of Lemma~\ref{lem:D0}, i.e., 
$D_0 y_{c} \notin E(K[c])_{\tors}$. For any ring class character $\chi : \Gal(K[c]/K) \ra \C^\times$ 
let 
$e_{\chi} \in \C[\Gal(K[c]/K)]$ be the eidempotent projector corresponding to $\chi$. Explicitly, 
$$
e_{\chi} = \frac{1}{\# \Gal(K[c]/K}) \sum_{\sigma \in \Gal(K[c]/K)} \chi^{-1}(\sigma) \sigma \in \C[\Gal(K[c]/K)]. 
$$
Consider $V = E(K[c]) \otimes \C$ as a complex representation of $\Gal(K[c]/K)$. Then the vector 
$D_0 y_c \otimes 1 \in V$ is nontrivial and since 
$$
V = \bigoplus_{\chi : \Gal(K[c]/K) \ra \C^\times} V_\chi, 
$$
then there exists a ring class character $\chi$, such that $e_\chi D_0(y_c \otimes 1) \ne 0$ 
(here, $V_\chi$ is the eigenspace corresponding to the character $\chi$). Next, we consider 
the point $D_0 D_c y_c \in E(K[c])$. 

Finally, we claim that $D_0 D_c y_c \otimes 1 \in E(K[c]) \otimes \C$ is 
nonzero, which is sufficient to conclude that $P_c = D_0 D_c y_c \notin E(K[c])_{\tors}$. We prove 
that $e_\chi (D_0 D_c y_c \otimes 1) \ne 0$. Indeed, 
\begin{eqnarray*}
e_\chi D_0 D_c (y_c \otimes 1) &=& e_{\chi} D_c D_0 (y_c \otimes 1) = \prod_{\ell \mid c} 
\left ( \sum_{i = 1}^\ell i \sigma_\ell^i \right ) e_{\chi} D_0(y_c \otimes 1) = \\ 
&=& \prod_{\ell \mid c} \left ( \sum_{i = 1}^\ell i \chi ( \sigma_\ell )^i \right ) e_{\chi} D_0(y_c \otimes 1),       
\end{eqnarray*}
the last equality holding since $\tau e_{\chi} = \chi(\tau) e_{\chi}$ in 
$\C[\Gal(K[c]/K)]$ for all 
$\tau \in \Gal(K[c]/K)$. Thus, it remains to compute $\ds \sum_{i=1}^\ell i \chi(\sigma_{\ell})^i$ for 
every $\ell \mid c$. It is not hard to show that 
$$
\sum_{i = 1}^\ell i \chi(\sigma_{\ell})^i = \left \{
\begin{array}{ll}
\frac{\ell + 1}{\chi(\sigma_\ell) - 1} & \textrm{if } \chi(\sigma_\ell) \ne 1 \\
\frac{\ell(\ell+1)}{2} & \textrm{if } \chi(\sigma_\ell) = 1. 
\end{array}
\right.  
$$
Thus, $e_\chi D_0 D_c (y_c \otimes 1) \ne 0$ which means that $P_c = D_0 D_c y_c \notin E(K[c])_{\tors}$ 
for any $c$ satisfying $D_0 y_c \notin E(K[c])_{\tors}$. To complete the proof, notice that for all, but finitely 
many $c \in \Lambda$, the hypothesis of Lemma~\ref{lem:D0} will be satisfied. 
\end{proof}

\subsection{Computational evidence for Kolyvagin's conjecture}\label{subsec:389A1}
Consider the example $E = \textrm{\bf 389A1}$ with equation 
{\small $y^2 + y = x^3 + x^2 - 2x$.}
As in Section~\ref{subsec:ptsexamples}, let $D = 7$, $\ell = 5$, and $p = 3$. Using the algorithm of~\cite[\S 2.1]{bsdalg1} we verify 
that the mod $p$ Galois representation $\rhobar_{E, p}$ is surjective. Next, we observe that $\ell = 5$ is a Kolyvagin prime for $E, p$ and $D$. 
Let $c = 5$ and consider the class $\kappa_{5, 1} \in \H^1(K, E[3])$. We claim that $\kappa_{5, 1} \ne 0$ which will verify Kolyvagin's conjecture. 

\begin{proposition}\label{prop:389-verify}
The class $\kappa_{5, 1} \ne 0$. In other words, Kolyvagin's conjecture holds for 
$E = \textrm{\bf 389A1}$,  
$D = 7$ and $p = 3$.  
\end{proposition} 

Before proving the proposition, we recall some standard facts about division polynomials (see, e.g.,~\cite[Ex.3.7]{silverman:aec}). 
For an elliptic curve given in Weierstrass form over any field of characteristic different from 2 and 3,   
{\small
$
y^2  = x^3 + Ax + B, 
$
}
one defines a sequence of polynomials $\psi_m \in \Z[A, B, x, y]$ inductively as follows\footnote{It is easy to check that these are polynomials.}: 
{\small 
\begin{eqnarray*}
& & \psi_1 = 1, \ \psi_2 = 2y, \\
& & \psi_3 = 3x^4 + 6Ax^2 + 12Bx - A^2, \\ 
& & \psi_4 = 4y(x^6 + 5Ax^4 + 20Bx^3 - 5A^2 x^2 - 4ABx - 8B^2 - A^3), \\
& & \psi_{2m+1} = \psi_{m+2}\psi_m^3 - \psi_{m-1}\psi_{m+1}^3 \ \ \textrm{for } m \geq 2, \\
& & 2y\psi_{2m} = \psi_m(\psi_{m+2}\psi_{m-1}^2  -  \psi_{m-2}\psi_{m+1}^2) \ \ \textrm{for } m \geq 3.  
\end{eqnarray*}
}
Define also polynomials $\phi_m$ and $\omega_m$ by 
{\small
$$
\phi_m = x\psi_m^2 - \psi_{m+1}\psi_{m-1}, \ 4y\omega_m = \psi_{m+2}\psi_{m-1}^2 - \psi_{m-2}\psi_{m+1}^2. 
$$
}
After replacing $y^2$ by $x^3 + Ax+B$, the polynomials $\phi_m$ and $\psi_m^2$ can be viewed as polynomials in $x$ with leading 
terms $x^{m^2}$ and $m^2 x^{m^2 - 1}$, respectively. Finally, multiplication-by-$m$ is given by 
{\small 
$$
mP = \left ( \frac{\phi_m(P)}{\psi_m(P)^2}, \frac{\omega_m(P)}{\psi_m(P)^3} \right ). 
$$ 
}
\begin{proof}[Proof of Proposition~\ref{prop:389-verify}]
We already computed the Heegner point $y_5$ on the model 
{\small 
$
y^2 + y = x^3 + x^2 - 2x
$
}
in Section~\ref{subsec:ptsexamples}. The Weierstrass model for $E$ is   
{\small 
$
y^2 = x^3 - 7/3x + 107/108,  
$ 
}
so $\ds A = -7/3$ and $\ds B = 107/108$. 
We now compute the point $P_5 = \sum_{i = 1}^5 i \sigma^i (y_5) \in E(K[5])$ on the 
Weierstrass model, where $\sigma$ is a generator of $\Gal(K[5]/K)$.
To show that $\kappa_{5, 1} \ne 0$ we need to check that there is no point $Q = (x, y)$, such that $3Q = P_5$. For the verification of this fact, we use the division polynomial 
$\psi_3$ and the polynomial $\phi_3$. Indeed, it follows from the recursive definitions 
that   
{\small
\begin{eqnarray*}
\phi_3(x) &=& x^9 - 12Ax^7 - 168Bx^6 + (30A^2 + 72B)x^5 - 168ABx^4 + \\
 &+& (36A^3 + 144AB - 96B^2)x^3 + 72A^2Bx^2 + \\ 
 &+& (9A^4 - 24A^2B + 96AB^2 + 144B^2)x + 8A^3B + 64B^3. 
\end{eqnarray*}
}
Consider the polynomial $g(x) = \phi_3(x) - X(P_5)\psi_3(x)^2$, 
where $X(P_5)$ is the $x$-coordinate of the point $P_5$ on the Weierstrass 
model. We factor $g(x)$ (which has degree 9) over the number 
field $K[5]$ and check that it is irreducible. In particular, there is no root 
of $g(x)$ in $K[5]$, i.e., there is no $Q \in E(K[5])$, such that $3Q = P_5$. 
Thus, $\kappa_{5,1} \ne 0$.     
\end{proof}

\begin{remark}
Using exactly the same method as above, we verify Kolyvagin's conjecture for the 
other two elliptic curves of rank two from Section~\ref{subsec:ptsexamples}. 
For both $E = \textrm{\bf 709A1}$ and $E = \textrm{\bf 718B1}$ we use $D = 7$, $p = 3$ 
and $\ell = 5$ (which are valid parameters), and verify that $\kappa_{5,1} \ne 0$ in 
the two cases. For completeness, we provide all the data of each computation
in the three examples in the files \verb,389A1.txt,, 
\verb,709A1.txt, and \verb,718A1.txt,.  
\end{remark}

\section{Non-trivial elements of the Shafarevich-Tate group}\label{sec:expl-sha}

Throughout the entire section, let $E_{/\Q}$ be a non-CM elliptic curve, 
$K = \Q(\sqrt{-D})$, where $D$ is a 
Heegner discriminant for $E$ such that the Heegner point $y_K$ has infinite order 
in $E(K)$ (which, by the Gross-Zagier formula and Kolyvagin's result, means that 
$E(K)$ has Mordell-Weil rank one) and let $p$ be a prime, such that 
$p \nmid DN$ and the mod $p$ Galois representation $\rhobar_{E, p}$ is surjective.  

\subsection{Non-triviality of Kolyvagin classes.}
Under the above assumptions, the next proposition provides a criterion which 
guarantees that an explicit class in the Shafarevich-Tate group 
$\Sha(E/K)$ is non-zero. 

\begin{proposition}\label{prop:nontriv}
Let $c \in \Lambda_m$. Assume that the following hypotheses are satisfied: 
\begin{enumerate}
\item{}[Selmer hypothesis]:\label{hyp:sel} The class $\kappa_{c, m}\in \H^1(K,E[p^m])$ 
is an element of the Selmer group $\Sel_{p^m}(E/K)$.

\vspace{0.1in}

\item{}[Non-divisibility]:\label{hyp:nondiv} The derived Heegner point $P_c$ is not divisible by $p^m$ 
in $E(K[c])$, i.e., $P_c \notin p^m E(K[c])$.   

\vspace{0.1in}

\item{}[Parity]:\label{hyp:parity} The number $f_c = \# \{ \ell : \ell \mid c \}$ is odd.  
\end{enumerate}
Then the image $\kappa'_{c, m}\in \H^1(K,E)[p^m]$ 
of $\kappa_{c,m}$ is a non-zero element of  $\Sha(E/K)[p^m]$.
\end{proposition}

\begin{proof}
The first hypothesis implies that the image $\kappa'_{c, m}$ of $\kappa_{c, m}$ in $\H^1(K, E)[p^m]$ is an element of 
the Shafarevich-Tate group $\Sha(E/K)$. The second one implies that $\kappa_{c, m} \ne 0$. To show that $\kappa'_{c, m} \ne 0$ 
we use the exact sequence 
$$
0 \ra E(K)/p^mE(K) \ra \Sel_{p^m}(E/K) \ra \Sha(E/K)[p^m] \ra 0
$$
which splits under the action of complex conjugation as 
$$
0 \ra ( E(K)/p^mE(K) )^{\pm} \ra \Sel_{p^m}(E/K)^{\pm} \ra \Sha(E/K)^{\pm}[p^m] \ra 0. 
$$
According to~\cite[Prop.5.4(2)]{gross:kolyvagin}, the class $\kappa_{c, m}$ lies in the $\eps_c$-eigenspace of the Selmer group $\Sel_{p^m}(E/K)$ for the 
action of complex conjugation, where $\eps_c = \eps (-1)^{f_c} = -1$ ($f_c$ is odd by the third hypothesis and $\eps = 1$ since $-\eps$ is the sign of the functional equation 
for $E_{/K}$ which is $-1$ by Gross-Zagier). On the other hand, the Heegner point $y_K = P_1$ lies  
in the $\eps_1$-eigenspace of complex conjugation (again, by~\cite[Prop.5.4(2)]{gross:kolyvagin}) where $\eps_1 = \eps (-1)^{f_1} = 1$. Since $E(K)$ has rank one, the group 
$E(K)^{-}$ is torsion and since $E(K)[p] = 0$, we obtain that $(E(K)/p^mE(K))^{-} = 0$. Therefore, 
$$
\Sel_{p^m}(E/K)^{-} \isom \Sha(E/K)^{-}[p^m], 
$$
which implies $\kappa'_{c, m} \ne 0$. 
\end{proof}

\subsection{The example $E = \textrm{\bf 53A1}$.}
The Weierstrass equation for the curve $E=${\bf \textrm{53A1}} is    
{\small 
$
y^2 = x^3 + 405 x + 16038  
$
}
and $E$ has rank one over $\Q$. The Fourier coefficient 
$a_5(f) \equiv 5 + 1 \equiv 0 \mod 3$, so 
$\ell = 5$ is a Kolyvagin prime for $E$, the discriminant $D = 43$ and the prime $p = 3$. 
Kolyvagin's construction exhibits a class $\kappa_{5, 1} \in \H^1(K, E[3])$. We will 
prove the following proposition:

\begin{proposition}\label{prop:53a1}
The cohomology class $\kappa_{5, 1} \in \H^1(K, E[3])$ lies in the Selmer group $\Sel_3(E/K)$ and its image $\kappa'_{5, 1}$ in 
the Shafarevich-Tate group $\Sha(E/K)$ is a nonzero $3$-torsion element.  
\end{proposition}

\begin{remark} 
Since $E/K$ has analytic rank one, Kolyvagin's conjecture
is automatic (since $m_0 < \infty$ by Gross-Zagier's formula) and one 
knows (see~\cite[Thm.~5.8]{mccallum:kolyvagin})
that there exist Kolyvagin classes $\kappa'_{c,m}$ which generate 
$\Sha(E/K)[p^\infty]$.  Yet, this result is not explicit in the sense that 
one does not know any particular Kolyvagin class which is non-trivial. 
The above proposition exhibits an explicit non-zero cohomology class 
in the $p$-primary part of the Shafarevich-Tate group $\Sha(E/K)$.  
\end{remark}

\begin{proof}
Using the data computed in Section~\ref{subsec:ptsexamples} for this curve, we 
apply the Kolyvagin derivative to compute the point $P_5$. In order to do 
this, one needs a generator of the Galois group $\Gal(K[5] / K)$. 
Such a generator is determined by the image of $\alpha$, which will be 
another root of $f(x)$ in $K[5]$. We check that the automorphism $\sigma$ 
defined by 
{\small 
\begin{eqnarray*}
\alpha & \mapsto & \frac{1}{1601320}(47343 + 54795\sqrt{-43})\alpha^5 + 
\frac{1}{2401980}(- 614771 -936861\sqrt{-43})\alpha^4 + \\  
& + & \frac{1}{600495}(34507457 + 40541607\sqrt{-43} )\alpha^3 + 
    \frac{1}{4803960}(102487877  -767102463\sqrt{-43} )\alpha^2 + \\
& + &  \frac{1}{400330}(- 61171198 + 52833377\sqrt{-43})\alpha + 
\frac{1}{200165}(18971815  -7453713\sqrt{-43})
\end{eqnarray*}
}
is a generator (we found this automorphism by factoring
the defining polynomial of the number field over the number
field $K[5]$). Thus, we can compute $\ds P_5 = \sum_{i = 1}^5 i \sigma^i(y_5)$. 

Note that we are computing the point on the Weierstrass model of $E$ 
rather than on the original model. The cohomology class $\kappa_{5, 1}$ is represented 
by the cocycle 
{\small 
$$
\ds \sigma \mapsto -\frac{(\sigma -1)P_5}{3} + \sigma\frac{P_5}{3}-\frac{P_5}{3}
$$
}
which is trivial if and only if 
$P_5 \in 3 E(K[5])$. To show that $P_5 \notin 3E(K[5])$ we repeat 
the argument of Proposition~\ref{prop:389-verify} and verify (using 
any factorization algorithm for polynomials over number fields) that  
the polynomial $g(x) = \phi_3(x) - X(P_5)\psi_3(x)^2$ has no linear factors over 
$K[5]$ (here, $X(P_5)$ is the $x$-coordinate of $P_5$). This means that there is no 
point $Q = (x, y) \in E(K[5])$, such that $3Q = P_5$, i.e., $\kappa_{5, 1} \ne 0$. 
Finally, using Proposition~\ref{prop:nontriv} we conclude that the class 
$\kappa'_{5, 1} \in \Sha(E/K)[3]$ is non-trivial. 
\end{proof}	

\begin{remark}
For completeness, all the computational data is provided (with the appropriate explanations) in the file \verb,53A1.txt,. 
We verified the irreducibility of $g(x)$ using MAGMA and PARI/GP independently.
\end{remark}

\bibliography{biblio}

\section{Appendix - Upper bounds on the logarithmic heights of the Heegner points $y_c$}
We explain how to compute an upper bound on the logarithmic height $h(y_c)$. The method 
first relates the canonical height $\widehat{h}(y_c)$ to special values of the first 
derivatives of certain automorphic $L$-functions via Zhang's generalization of the Gross-Zagier formula. Then we either 
compute the special values up to arbitrary precision using a well-known 
algorithm (recently implemented by Dokchitser) or use effective asymptotic upper 
bounds (convexity bounds) on the special values and Cauchy's integral formula. 
Finally, using some known bounds on the difference between the canonical and the logarithmic heights, we obtain explicit upper bounds on the logarithmic height 
$h(y_c)$. We provide a summary of the asymptotic bounds in 
Section~\ref{subsec:asymptotic} and refer the reader to~\cite{jetchev:asymptotic} 
for complete details.     

\subsection{The automorphic $L$-functions $L(f, \chi, s)$ and $L(\pi, s)$}
\label{subsec:lfunc} 

Let $d_c = c^2 D$ and let $\ds f = \sum_{n \geq 1} a_n q^n$ be the new eigenform of 
level $N$ and weight two corresponding to $E$. Let $\chi : \Gal(K[c]/K) \ra \C^\times$ 
be a ring class character. 

\vspace{0.1in}

\noindent \textit{1. The theta series $\theta_\chi$.} Recall that ideal classes 
for $\Pic(\cO_c)$ correspond to primitive, reduced 
binary quadratic forms of discriminants $d_c$. To each ideal class $\cA$ we 
consider the corresponding binary quadratic form $Q_{\cA}$ and the theta series 
$\theta_{Q_{\cA}}$ associated to it via 
$$
\theta_{Q_{\cA}} = \sum_{M} e^{2\pi i z Q_{\cA}(M)}
$$
which is a modular form for $\Gamma_0(d_c)$ of weight one with character $\eps$ (the quadratic character of $K$)  
according to Weil's converse theorem (see~\cite{shimura:jac_factors} for 
details). This allows us to define a cusp form  
$$
\theta_\chi = \sum_{\cA \in \Pic(\cO_c)} \chi^{-1}(\cA) \theta_{Q_{\cA}} \in S_1(\Gamma_0(d_c), \eps). 
$$
Here, we view $\chi^{-1}$ as a character of $\Pic(\cO_c)$ via the 
isomorphism $\Pic(\cO_c) \isom \Gal(K[c]/K)$. Let $\theta_\chi = \sum_{m \geq 0} 
b_m q^m$ be the Fourier expansion. By $L(f, \chi, s)$ we will always mean the 
Rankin $L$-function\edit{Put a reference for Rankin $L$-functions!} $L(f \otimes \theta_\chi, s)$ (equivalently, the $L$-function 
associated to the automorphic representation $\pi = f \otimes \theta_\chi$ of 
$\GL_2$).    

\vspace{0.1in}

\noindent \textit{2. The functional equation of $L(f, \chi, s)$.}
We recall some basic facts about the Rankin $L$-series $L(f \otimes \theta_\chi, s)$ 
following~\cite[\S III]{gross:heegner_points}. Since $(N, D)=1$, the {\em conductor} of 
$L(f \otimes \theta_\chi, s)$ is 
$Q = N^2 d_c^2$. The Euler factor at infinity (the gamma factor) is 
$L_\infty(f \otimes \theta_{\chi}, s) = \Gamma_\C(s)^2$. If we set 
$$
\Lambda(f\otimes \theta_{\chi}, s) = Q^{s/2} L_\infty(f \otimes \theta_\chi, s) L(f\otimes \theta_\chi, s) 
$$ 
then the function $\Lambda$ has a holomorphic continuation to the entire complex plane 
and satisfies the functional equation 
$$
\Lambda(f\otimes \theta_\chi, s) = -\Lambda(f \otimes \theta_\chi, 2-s). 
$$
In particular, the order of vanishing of $L(f\otimes \theta_\chi, s)$ at $s = 1$ is non-negative and odd, i.e., $L(f \otimes \theta_\chi, 1) = 0$. 

\vspace{0.1in}

\noindent \textit{3. The shifted $L$-function $L(\pi, s)$.}
In order to center the critical line at $\ds \Re(s) = \frac{1}{2}$ instead of 
$\ds \Re(s) = 1$ (which is consistent with Langlands convention), we will be 
looking at the shifted automorphic $L$-function
$$
L(\pi, s) = L \left (f \otimes \theta_\chi, s+\frac{1}{2} \right )
$$ 
Moreover, $L(\pi, s)$ satisfies a functional equation relating the values at 
$s$ and $1-s$. Let 
$$
L(\pi, s) = \sum_{n \geq 1} \frac{\lambda_\pi(n)}{n^{s}} = 
\prod_p (1 - \alpha_{\pi, 1}(p)p^{-s})^{-1} \dots (1 - \alpha_{\pi, d}(p)p^{-s})^{-1}  
$$
be the Dirichlet series and the Euler product of $L(\pi, s)$ 
(which are absolutely convergent for $\Re(s) > 1$).

\subsection{Zhang's formula}
For a character $\chi$ of $\Gal(K[c]/K)$, let 
$$
e_\chi = \frac{1}{\# \Gal(K[c]/K)} \sum_{\sigma \in \Gal(K[c]/K)} \chi^{-1}(\sigma) \sigma \in \C[\Gal(K[c]/K)]
$$
be the associated eidempotent. The canonical height $\widehat{h}( e_\chi y_c )$ is related via the generalized Gross-Zagier formula 
of Zhang to a special value of the derivative of the $L$-function 
$L(f, \chi, s)$ at $s = 1$ (see~\cite[Thm.1.2.1]{zhang:gross-zagier}). More precisely, 

\begin{theorem}[Zhang]\label{thm:zhang}
If $(\,,)$ denotes the Petersson inner product on $S_2(\Gamma_0(N))$ then 
$$
L'(f, \chi, 1) = \frac{4}{\sqrt{D}} (f, f) \widehat{h} (e_\chi y_c ). 
$$
\end{theorem} 

\noindent Since $\langle e_{\chi'}y_c, e_{\chi''} y_c\rangle = 0$ whenever $\chi' \ne \chi''$ (here, $\langle \,,\rangle$ denotes the N\'eron-Tate 
height pairing for $E$) and since $\widehat{h}(x) = \langle x, x\rangle $ then 
\begin{equation}\label{eqn:height}
\widehat{h}(y_c) = \widehat{h} \left (\sum_{\chi} e_\chi y_c \right ) = \sum_{\chi} \widehat{h}(e_\chi y_c).  
\end{equation}
Thus, we will have an upper bound on the canonical height $\widehat{h}(y_c)$ if 
we have upper bounds on the special values $L'(f, \chi, 1)$ for every character 
$\chi$ of $\Gal(K[c]/K)$.   

\subsection{Computing special values of derivatives of automorphic $L$-functions}
For simplicity, let $\gamma(s) = L_\infty(f \otimes \theta_\chi, s+1/2)$ be the 
gamma factor of the $L$-function $L(\pi, s)$. This means that if 
$\lambda(\pi, s) = Q^{s/2} \gamma(s) L(\pi, s)$ then $\Lambda(\pi, s)$ satisfies the 
functional equation $\Lambda(\pi, s) = \Lambda(\pi, 1-s)$. We will describe a 
classical algorithm to compute the value of $L^{(k)}(\pi, s)$ at $s = s_0$ up to arbitrary precision. The algorithm and its implementation is discussed in a greater generality in~\cite{dokchitser}. The main idea is to express $\Lambda(\pi, s)$ as an infinite series 
with rapid convergence which is usually done in the following sequence of steps:  

\begin{enumerate}
\item Consider the inverse Mellin transform of the gamma factor $\gamma(s)$, i.e., 
the function $\phi(t)$ which satisfies 
$$
\gamma(s) = \int_0^\infty \phi(t) t^s \frac{dt}{t}. 
$$
One can show (see~\cite[\S 3]{dokchitser}) that $\phi(t)$ decays exponentially for 
large $t$. Hence, the sum 
$$
\Theta(t) = \sum_{n = 1}^\infty \lambda_\pi(n) \phi \left (\frac{nt}{\sqrt{Q}} \right ) 
$$
converges exponentially fast. The function $\phi(t)$ can be computed numerically as 
explained in~\cite[\S 3-5]{dokchitser}. 

\item The Mellin transform of $\Theta(t)$ is exactly the function $\Lambda(\pi, s)$. 
Indeed,  
\begin{eqnarray*}
\int_0^\infty \Theta(t) t^s \frac{dt}{t} &=& 
\int_0^\infty \sum_{n = 1}^\infty \lambda_{\pi}(n) \phi \left (\frac{nt}{\sqrt{Q}} \right ) 
t^s \frac{dt}{t} = \sum_{n=1}^\infty \lambda_{\pi}(n) \int_0^\infty 
\phi \left (\frac{nt}{\sqrt{Q}} \right ) t^s \frac{dt}{t} = \\  
& = & \sum_{n = 1}^\infty \lambda_{\pi}(n) \left ( \frac{\sqrt{Q}}{n} \right )^s \int_0^\infty \phi(t') t'^s \frac{dt'}{t'} = Q^{s/2} \gamma(s) L(\pi, s) = \Lambda(\pi, s).  
\end{eqnarray*}

\item Next, we obtain a functional equation for $\Theta(t)$ which relates 
$\Theta(t)$ to $\Theta(1/t)$. Indeed, since $\Lambda(\pi, s)$ is holomorphic, Mellin's 
inversion formula implies that  
$$
\Theta(t) = \int_{c - i \infty}^{c + i \infty} \Lambda(\pi, s) t^{-s} ds,\ \forall c.  
$$
Therefore, 
\begin{eqnarray*}
\Theta(1/t) &=& \int_{c - i\infty}^{c+i\infty} \Lambda(\pi, s) (1/t)^{-s} ds = - 
t\int_{c - i\infty}^{c+i\infty} \Lambda(\pi, 1-s) t^{-(1-s)} ds = \\
&=& -t \int_{c - i\infty}^{c+i\infty} \Lambda(\pi, s') t^{-s'}ds' = -t \Theta(t). 
\end{eqnarray*}
Thus, $\Theta(t)$ satisfies the functional equation $\Theta(1/t) = -t \Theta(t)$. 

\item Next, we consider the incomplete Mellin transform 
$$
G_s(t) = t^{-s} \int_t^\infty \phi(x) x^s \frac{dx}{x}, \ t > 0
$$ 
of $\phi(t)$. The function $G_s(t)$ satisfies $\ds \lim_{t \ra 0} t^s G_s(t) = \gamma(s)$ 
and it decays exponentially. Moreover, it can be computed numerically 
(see~\cite[\S 4-5]{dokchitser}).  

\item Finally, we use the functional equation for $\Theta(t)$ to obtain 
\begin{eqnarray*}
\Lambda(\pi, s) &=& \int_0^\infty \Theta(t) t^s \frac{dt}{t} = \int_0^1  
\Theta(t) t^s \frac{dt}{t} + \int_1^{\infty} \Theta(t) t^s \frac{dt}{t} = \\
&=& \int_1^\infty \Theta(1/t') t'^{-s} \frac{dt'}{t'} + 
\int_1^\infty \Theta(t) t^s \frac{dt}{t} = \\
&=& -\int_1^\infty \Theta(t') t'^{1-s} \frac{dt'}{t'} + \int_1^\infty \Theta(t)t^s 
\frac{dt}{t}.  
\end{eqnarray*}

\item Finally, we compute 
\begin{eqnarray*}
\int_1^\infty \Theta(t)t^{s} \frac{dt}{t} &=& 
\int_1^\infty \sum_{n = 1}^\infty \lambda_{\pi}(n) \phi \left ( \frac{nt}{\sqrt{Q}}\right )
t^s \frac{dt}{t} = 
\sum_{n=1}^\infty \lambda_{\pi}(n) \int_1^\infty \phi \left ( \frac{nt}{\sqrt{Q}}\right ) t^s 
\frac{dt}{t} = \\
&=& \sum_{n=1}^\infty \lambda_{\pi}(n) \int_{\frac{n}{\sqrt{Q}}}^\infty \phi \left ( t'\right ) \left ( \frac{\sqrt{Q}t'}{n} \right )^s =  
\sum_{n=1}^\infty \lambda_{\pi}(n) G_s \left (\frac{n}{\sqrt{Q}} \right). 
\end{eqnarray*}
Thus, 
$$
\Lambda(\pi, s) = \sum_{n=1}^\infty \lambda_{\pi}(n) G_s \left (\frac{n}{\sqrt{Q}} \right) 
- \sum_{n=1}^\infty \lambda_{\pi}(n) G_{1-s} \left (\frac{n}{\sqrt{Q}} \right)
$$
is the desired expansion. From here, we obtain a formula for the $k$-th derivative 
$$
\frac{\partial^k}{\partial s^k} \Lambda(\pi, s) = 
\sum_{n=1}^\infty \lambda_{\pi}(n) \frac{\partial^k}{\partial s^k}G_s \left (\frac{n}{\sqrt{Q}} \right) 
- \sum_{n=1}^\infty \lambda_{\pi}(n) \frac{\partial^k}{\partial s^k}G_{1-s} \left (\frac{n}{\sqrt{Q}} \right). 
$$
The computation of the derivatives of $G_s(x)$ is explained in~\cite[\S 3-5]{dokchitser}. 
\end{enumerate}

\subsection{Asymptotic estimates on the canonical heights $\widehat{h}(y_c)$}\label{subsec:asymptotic}

In this section we provide an asymptotic bound on the canonical height 
$\widehat{h}(y_c)$ by using convexity bounds on the special values of the 
automorphic $L$-functions $L(\pi, s)$ defined in Section~\ref{subsec:lfunc}. 
We only outline the basic techniques used to prove the asymptotic bounds and 
refer the reader to~\cite{jetchev:asymptotic} for the complete details.
Asymptotic bounds on heights of Heegner points are obtained 
in~\cite{ricotta-vidick}, but these bounds are of significantly different type than 
ours. In our case, we fix the elliptic curve $E$ and let the fundamental discriminant 
$D$ and the conductor $c$ of the ring class field both vary. The 
result that we obtain is the following 

\begin{proposition}
Fix the elliptic curve $E$ and let the fundamental discriminant $D$ and the conductor 
$c$ vary. For any $\eps > 0$ the following asymptotic bound holds 
$$
\widehat{h}(y_c) \ll_{\eps, f} h_D D^{\eps} c^{2+\eps},  
$$
where $h_D$ is the class number of the quadratic imaginary field $K = \Q(\sqrt{-D})$. 
Moreover, the implied constant depends only on $\eps$ and the cusp form $f$. 
\end{proposition}
  
One proves the proposition by combining the formula of Zhang with convexity 
bounds on special values of automorphic $L$-functions. 
The latter are conveniently expressed in terms of a quantity 
known as the {\em analytic conductor} associated to the automorphic representation 
$\pi$ (see~\cite[p.12]{michel:parkcity}). It is a function $Q_\pi(t)$ over the real 
line, which is defined as 
$$
Q_\pi(t) = Q \cdot \prod_{i = 1}^d (1 + |it - \mu_{\pi, i}|), \ \forall t\in \R,  
$$ 
where $\mu_{\pi, i}$ are obtained from the gamma factor 
$$
L_\infty(\pi, s) = \prod_{i = 1}^d \Gamma_\R(s - \mu_{\pi, i}), \ \Gamma_\R(s) = 
\pi^{-s/2}\Gamma(s/2). 
$$
In our situation, $d = 4$ and $\mu_{\pi, 1} = \mu_{\pi, 2} = 0$, 
$\mu_{\pi, 3} = \mu_{\pi, 4} = 1$ (see~\cite[\S 1.1.1]{michel:parkcity} 
and~\cite[\S 3]{serre:zeta} for discussions of local factors at archimedian 
places). Moreover, we let $Q_\pi = Q_\pi(0)$.

The main idea is to prove that for a fixed $f$, 
$|L'(\pi_{f \otimes \theta_\chi}, 1/2)| \ll_{\eps, f} Q_{\pi_{f \otimes \theta_\chi}}^{1/4 + \eps}$, where the implied constant only depends on $f$ and $\eps$ (and is independent of $\chi$ and the 
discriminant $D$). To establish the bound, we first prove an asymptotic bound for 
the $L$-function $L(\pi_{f\otimes \theta_\chi}, s)$ on 
the vertical line $\Re(s) = 1 + \eps$ by either using the Ramanujan-Petersson 
conjecture or a method of Iwaniec (see~\cite[p.26]{michel:parkcity}). This gives us the 
estimate   
$|L(\pi_{f \otimes \theta_\chi}, 1+\eps + it)| \ll_{\eps, f} Q_{\pi_{f \otimes \theta_\chi}}(t)^\eps$. Then, by the functional equation 
for $L(\pi_{f \otimes \theta_\chi}, s)$ and Stirling's approximation formula, we deduce an upper bound for the 
$L$-function on the vertical line $\Re(s) = -\eps$, i.e., 
$|L(\pi_{f \otimes \theta_\chi}, -\eps + t)| 
\ll_{\eps, f} Q_{\pi_{f \otimes \theta_\chi}}(t)^{1/2 + \eps}$. Next, we apply Phragmen-Lindel\"of's convexity principle 
(see~\cite[Thm.5.53]{iwaniec-kowalski}) to obtain the bound
$|L(\pi_{f \otimes \theta_\chi}, 1/2 + it)| \ll_{\eps, f} Q_\pi(t)^{1/4+\eps}$ (also known as {\em convexity bound}). Finally, by applying Cauchy's integral formula for a 
small circle centered at 
$s = 1/2$, we obtain the asymptotic estimate 
$|L'(\pi_{f \otimes \theta_\chi}, 1/2)| \ll_{\eps, f} Q_{\pi_{f \otimes \theta_\chi}}^{1/4+\eps}$.      
Since $Q= N^2 d_c^2 = N^2 D^2 c^4$ in our situation and since 
$[K[c]:K] = h_D \prod_{\ell \mid c} (\ell+1)$, Zhang's formula (Theorem~\ref{thm:zhang}) and equation (\ref{eqn:height}) imply that for any $\eps > 0$, 
$$
\widehat{h}(y_c) \ll_{\eps, f} h_D D^{\eps} c^{2+\eps}.  
$$

\begin{remark}
In the above situation (the Rankin-Selberg $L$-function of two cusp forms 
of levels $N$ and $d_c = c^2D$), one can even prove a subconvexity bound   
$|L'(\pi_{f \otimes \theta_\chi}, 1/2)| \ll_{f} D^{1/2 - 1/1057} c^{1 - 2/1057}$, where 
the implied constant depends only on $f$ and is independent of $\chi$ 
(see~\cite[Thm.2]{michel:subconvexity}). Yet, the proof relies on much more involved 
analytic number theory techniques than the convexity principle, so we do not discuss it here.       
\end{remark}

\subsection{Height difference bounds and the main estimates}\label{subsec:heightbds}
To estimate $h(y_c)$ we need a bound 
on the difference between the canonical and the logarithmic heights. Such a bound 
has been established in~\cite{silverman:height} and~\cite{cremona-siksek} and is effective. 

Let $F$ be a number field. For any non-archimedian place $v$ of $K$, let $E^0(F_v)$ denote 
the points of $E(F_v)$ which specialize to the identity component of the N\'eron model of 
$E$ over the ring of integers $\cO_v$ of $F_v$. Moreover, let $n_v = [F_v : \Q_v]$ and let 
$M_F^\infty$ denote the set of all archimedian places of $F$. A slightly weakened 
(but easier to compute) bounds on the height difference are provided by the 
following result of~\cite[Thm.2]{cremona-siksek} 

\begin{theorem}[Cremona-Prickett-Siksek]
Let $P \in E(F)$ and suppose that $P \in E^0(F_v)$ for every non-archimedian place $v$ 
of $F$. Then 
$$
\frac{1}{3[F : \Q]} \sum_{v \in M_F^\infty} n_v \log \delta_v \leq h(P) - \widehat{h}(P) 
\leq \frac{1}{3[F : \Q]} \sum_{v \in M_F^\infty} n_v \log \eps_v, 
$$ 
where $\eps_v$ and $\delta_v$ are defined in~\cite[\S 2]{cremona-siksek}. 
\end{theorem}  

\begin{remark}
All of the points $y_c$ in our particular examples satisfies the condition 
$y_c \in E^0(K[c]_v)$ for all non-archimedian places $v$ of $K[c]$. Indeed, according 
to~\cite[\S III.3]{gross-zagier} (see also~\cite[Cor.3.2]{jetchev:tamagawa}) the 
point $y_c$ lies in $E^0(K[c]_v)$ up to a rational torsion point\footnote{See also~\cite{jetchev:tamagawa} for 
another application of this local property of the points $y_c$.}. Since 
$E(\Q)_{\tor}$ is trivial for all the curves that we are considering, the above 
proposition is applicable. In general, one does not need this assumption in order 
to compute height bounds (see~\cite[Thm.1]{cremona-siksek} for the general case).    
\end{remark}

\begin{remark}
A method for computing $\eps_v$ and $\delta_v$ up to arbitrary precision for real 
and complex archimedian places is provided in~\cite[\S 7-9]{cremona-siksek}.  
\end{remark}

\end{document}

%% file: macros.tex
\usepackage[active]{srcltx}\usepackage[active]{srcltx}\usepackage[active]{srcltx}\usepackage[active]{srcltx}
\usepackage[centertags]{amsmath}
\usepackage{amsfonts}
\usepackage{amssymb}
\usepackage{amsthm}
\usepackage{amscd}

\newcommand{\edit}[1]{\footnote{#1}\marginpar{\hfill {\sf\thefootnote}}}

\newcommand{\cN}{\mathcal N}

\DeclareMathOperator{\MW}{MW}

\DeclareMathOperator{\Pic}{Pic}

\newcommand{\an}{{\rm an}}

\DeclareMathOperator{\Sel}{Sel}

\renewcommand{\Re}{\mbox{\rm Re}}
\DeclareMathOperator{\Imm}{Im}
\renewcommand{\Im}{\Imm}

\DeclareFontEncoding{OT2}{}{} 
  \newcommand{\textcyr}[1]{%
    {\fontencoding{OT2}\fontfamily{wncyr}\fontseries{m}\fontshape{n}%
     \selectfont #1}}
\newcommand{\Sha}{{\mbox{\textcyr{Sh}}}}

\newcommand{\cA}{\mathcal{A}}

\newcommand{\cG}{\mathcal{G}}

\newcommand{\cO}{\mathcal{O}}

\newcommand{\ds}{\displaystyle}

\DeclareMathOperator{\tor}{tor}

\DeclareMathOperator{\tors}{tors}

\newcommand{\ra}{\rightarrow}

\newcommand{\hra}{\hookrightarrow}

\newcommand{\eps}{\varepsilon}
\newcommand{\vphi}{\varphi}

\newcommand{\comment}[1]{}
\newcommand{\Q}{\mathbb{Q}}

\newcommand{\R}{\mathbb{R}}

\newcommand{\C}{\mathbb{C}}

\newcommand{\Qbar}{\overline{\Q}}

\newcommand{\Z}{\mathbb{Z}}
\newcommand{\F}{\mathbb{F}}

\newcommand{\isom}{\cong}
\DeclareMathOperator{\ab}{ab}
\DeclareMathOperator{\Aut}{Aut}

\DeclareMathOperator{\ord}{ord}
\DeclareMathOperator{\GL}{GL}

\DeclareMathOperator{\Gal}{Gal}

\newcommand{\rhobar}{\overline{\rho}}

\newcommand{\h}{\mathfrak{h}}

\DeclareMathOperator{\Tr}{Tr}

\DeclareMathOperator{\HH}{H}
\renewcommand{\H}{\HH}

\theoremstyle{plain}
\newtheorem{theorem}{Theorem}[section]
\newtheorem{proposition}[theorem]{Proposition}
\newtheorem{corollary}[theorem]{Corollary}

\newtheorem{lemma}[theorem]{Lemma}

\newtheorem{conjecture}[theorem]{Conjecture}

\theoremstyle{definition}

\newtheorem{alg}[theorem]{Algorithm}

\newtheorem{remark}[theorem]{Remark}

\numberwithin{equation}{section}
\numberwithin{figure}{section}
\numberwithin{table}{section}

\newcounter{listnum}

{\end{alg}}
{\begin{enumerate}\setlength{\itemsep}{0.1ex}}{\end{enumerate}}

